\documentclass[a4paper,11pt]{amsart}

\usepackage{epsf,graphics,amssymb,amsmath,xcolor,asymptote,microtype}
\usepackage{hyperref}
\usepackage[justification=centering]{caption}
\usepackage[nameinlink,nosort]{cleveref}

\Crefname{subsection}{Section}{Sections}
\Crefname{prop}{Proposition}{Propositions}
\crefname{equation}{}{}
\let\cref\Cref

\begin{asydef}
real strokewidth = 0.7;
transform myt = scale(1, 1.66);
pen braidpen = linewidth(strokewidth);

void mydraw(path p) {
    draw(myt * p, braidpen);
}

// x = some crossing
// 0 = the crossing in question
// o = no crossing

// do
// .....
// ...o.
// ..0..
// .x...
// .....

// except
// .x...
// ..x..
// ..0..
// .x...
// .....

// do
// ..x..
// ...x.
// ...x.
// ..0..
// .x...
// .....

// do
// .....
// .o...
// ..0..
// ...x.
// .....
real howFarOnTop(int[][] B, int x, int y) {
     if (( ((B[x + 1][y] == 0) && (B[x][y - 1] != 0) && (B[x - 1][y - 2] != 0) && ((B[x - 1][y + 1] == 0) || (B[x][y] == 0)))  || ((B[x + 1][y] != 0) && (B[x][y - 1] != 0) && (B[x - 1][y - 2] != 0) && (B[x + 1][y + 1] != 0) && (B[x][y + 2] != 0))) ||
         ( ((B[x][y] == 0) && (B[x + 1][y - 1] != 0) && (B[x + 2][y - 2] != 0) && ((B[x + 2][y + 1] == 0) || (B[x + 1][y] == 0)))  || ((B[x][y] != 0) && (B[x + 1][y - 1] != 0) && (B[x + 2][y - 2] != 0) && (B[x][y + 1] != 0) && (B[x + 1][y + 2] != 0))))
     {
        return 0.15;
        dot(myt * (x,y), scale(1.5) * green);
    }
    return 0;
}

real howFarOnBottom(int[][] B, int x, int y) {
     if (( ((B[x + 1][y - 1] == 0) && (B[x][y] != 0) && (B[x - 1][y + 1] != 0) && ((B[x - 1][y - 2] == 0) || (B[x][y - 1] == 0)))  || ((B[x + 1][y - 1] != 0) && (B[x][y] != 0) && (B[x - 1][y + 1] != 0) && (B[x + 1][y - 2] != 0) && (B[x][y - 3] != 0))) ||
         ( ((B[x][y - 1] == 0) && (B[x + 1][y] != 0) && (B[x + 2][y + 1] != 0) && ((B[x + 2][y - 2] == 0) || (B[x + 1][y - 1] == 0)))  || ((B[x][y - 1] != 0) && (B[x + 1][y] != 0) && (B[x + 2][y + 1] != 0) && (B[x][y - 2] != 0) && (B[x + 1][y - 3] != 0))))
        return -0.15;
    return 0;
}

path hidingDisk(int x, int y) {
    return shift(x - .5, y + .5) * scale(0.2) * inverse(myt) * unitcircle;
}

void drawToLeft(int[][] B, int x, int y, pen braidpen, bool gap) {
    pair outgoingAngle = ((B[x + 1][y - 1] == 0) || ((B[x + 1][y - 1] != 0) && (B[x][y + 1] != 0) && (B[x + 1][y + 2] != 0)) || ((B[x + 1][y - 1] != 0) && (B[x + 1][y - 2] != 0) && (B[x][y - 3] != 0))) ? up : up + left;
    pair incomingAngle = ((B[x - 1][y + 1] == 0) || ((B[x - 1][y + 1] != 0) && (B[x - 1][y + 2] != 0) && (B[x][y + 3] != 0)) || ((B[x - 1][y + 1] != 0) && (B[x][y - 1] != 0) && (B[x - 1][y - 2] != 0))) ? up : up + left;

    pair midpoint = (x - .5, y + .5);
    path p;
    if (howFarOnBottom(B, x, y) != 0)
        p = (x,y + howFarOnBottom(B, x, y) + howFarOnTop(B, x, y)){outgoingAngle}..(x - .2, y + .2)..midpoint..{incomingAngle}(x - 1, y + 1 + howFarOnTop(B, x - 1, y + 1) + howFarOnBottom(B, x - 1, y + 1));
    else if (howFarOnTop(B, x - 1, y + 1) != 0)
        p = (x,y + howFarOnBottom(B, x, y) + howFarOnTop(B, x, y)){outgoingAngle}..midpoint..(x - .8, y + .8)..{incomingAngle}(x - 1, y + 1 + howFarOnTop(B, x - 1, y + 1) + howFarOnBottom(B, x - 1, y + 1));
    else
        p = (x,y + howFarOnBottom(B, x, y) + howFarOnTop(B, x, y)){outgoingAngle}..midpoint..{incomingAngle}(x - 1, y + 1 + howFarOnTop(B, x - 1, y + 1) + howFarOnBottom(B, x - 1, y + 1));
    if (gap) {
        real[][] i = intersections(p, hidingDisk(x, y));
        draw(myt * subpath(p, 0, i[0][0]), braidpen);
        draw(myt * subpath(p, i[i.length - 1][0], 3), braidpen);
    } else
        draw(myt * p, braidpen);
}

void drawToRight(int[][] B, int x, int y, pen braidpen, bool gap) {
    pair outgoingAngle = ((B[x][y - 1] == 0) || ((B[x][y - 1] != 0) && (B[x + 1][y + 1] != 0) && (B[x][y + 2] != 0)) || ((B[x][y - 1] != 0) && (B[x][y - 2] != 0) && (B[x + 1][y - 3] != 0))) ? up : up + right;
    pair incomingAngle = ((B[x + 2][y + 1] == 0) || ((B[x + 2][y + 1] != 0) && (B[x + 2][y + 2] != 0) && (B[x + 1][y + 3] != 0)) || ((B[x + 2][y + 1] != 0) && (B[x + 1][y - 1] != 0) && (B[x + 2][y - 2] != 0))) ? up : up + right;

    pair midpoint = (x + .5, y + .5);
    path p;
    if (howFarOnBottom(B, x, y) != 0)
        p = (x,y + howFarOnBottom(B, x, y) + howFarOnTop(B, x, y)){outgoingAngle}..(x + .2, y + .2)..midpoint..{incomingAngle}(x + 1, y + 1 + howFarOnTop(B, x + 1, y + 1) + howFarOnBottom(B, x + 1, y + 1));
    else if (howFarOnTop(B, x + 1, y + 1) != 0)
        p = (x,y + howFarOnBottom(B, x, y) + howFarOnTop(B, x, y)){outgoingAngle}..midpoint..(x + .8, y + .8)..{incomingAngle}(x + 1, y + 1 + howFarOnTop(B, x + 1, y + 1) + howFarOnBottom(B, x + 1, y + 1));
    else
        p = (x,y + howFarOnBottom(B, x, y) + howFarOnTop(B, x, y)){outgoingAngle}..midpoint..{incomingAngle}(x + 1, y + 1 + howFarOnTop(B, x + 1, y + 1) + howFarOnBottom(B, x + 1, y + 1));
    if (gap) {
        real[][] i = intersections(p, hidingDisk(x + 1, y));
        draw(myt * subpath(p, 0, i[0][0]), braidpen);
        draw(myt * subpath(p, i[i.length - 1][0], 3), braidpen);
    } else
        draw(myt * p, braidpen);
}

void drawbraid(int[] b, int strands = -1, pen[] strandPens = {}) {
    // delete zeroes
    for (int i = 0; i < b.length;) {
        if (b[i] == 0)
            b.delete(i);
        else
            ++i;
    }

    // get right number of strands
    if (strands == -1) {
        int m = 1;
        for (int i = 0; i < b.length; ++i)
            if (abs(b[i]) > m)
                m = abs(b[i]);
        strands = m + 1;
    }

    // slide, prepare matrix
    int[][] B = array(strands + 3, array(b.length + 4, 0));
    for (int i = 0; i < b.length; ++i) {
        int x = abs(b[i]) + 1;
        int y = b.length + 3;
        while ((B[x - 1][y - 1] == 0) && (B[x][y - 1] == 0) && (B[x + 1][y - 1] == 0) && (y > 2))
            --y;
        B[x][y] = (b[i] > 0) ? 1 : -1;
    }

    // until where are there crossings?
    int highestY = 0;
    for (int y = 2; y < B[0].length; ++y) {
        bool thereAreCrossings = false;
        for (int x = 1; x < strands + 1; ++x)
            if (B[x][y] != 0) {
                thereAreCrossings = true;
                break;
            }
        if (! thereAreCrossings) {
            highestY = y;
            break;
        }
    }

    //grid
/*    for (int x = 1; x < strands + 1; ++x)
        draw(myt * ((x, 1)--(x, highestY)), lightgray);
    for (int y = 1; y < highestY; ++y)
        draw(myt * ((1, y)--(strands, y)), lightgray);
*/

    // follow the strands
    int[] strandNumbers;
    for (int x = 0; x < strands; ++x)
        strandNumbers.push(x);

    // draw
    for (int y = 2; y < highestY; ++y)
        for (int x = 1; x < strands + 1; ++x) {
            pen leftPen = braidpen;
            pen rightPen = braidpen;
            if (strandPens.length != 0) {
                leftPen = strandPens[strandNumbers[x - 1]];
                if (x != 1)
                    rightPen = strandPens[strandNumbers[x - 2]];
            }
            if (B[x][y] != 0) {
                drawToRight(B, x - 1, y, rightPen, (B[x][y] < 0));
                drawToLeft(B, x, y, leftPen, (B[x][y] > 0));

                int tmp = strandNumbers[x - 2];
                strandNumbers[x - 2] = strandNumbers[x - 1];
                strandNumbers[x - 1] = tmp;
            }
            else if (B[x + 1][y] == 0)
                draw(myt * ((x, y + howFarOnTop(B, x, y))--(x, y + 1 + howFarOnBottom(B, x, y + 1))), leftPen);
        }

    // cut off above and below
    clip(myt * ((0,2)--(2 + strands, 2)--(2 + strands, highestY)--(0, highestY)--cycle));
} 

void drawbraid(string s, int strands = -1, pen[] strandPens = {}) {
    int[] b = {};
    for(int i = 0; i < length(s); ++i) {
        int c = 0;
        int a = ascii(substr(s, i));
        if ((a >= 97) && (a <= 122))
            c = a - 96;
        else if ((a >= 65) && (a <= 90))
            c = -(a - 64);
        if (c != 0)
            b.push(c);
    }
    drawbraid(b, strands, strandPens);
}

//unitsize(1cm);
//drawbraid("abcddcbaabcd");
\end{asydef}

\begin{asydef}
unitsize(3.5mm);
strokewidth = 0.6;
\end{asydef}
\hypersetup{colorlinks={true},linkcolor={black},citecolor={black},filecolor={black},urlcolor={black},pdfauthor={S. Baader, P. Feller, L. Lewark, R. Zentner},pdftitle={Khovanov width and dealternation number of positive braid links}}

\theoremstyle{plain}
\newtheorem{theorem}{Theorem}
\newtheorem{corollary}[theorem]{Corollary}
\newtheorem{lemma}[theorem]{Lemma}
\newtheorem{proposition}[theorem]{Proposition}

\theoremstyle{remark}

\newtheorem*{remark}{Remark}

\def\et{\;\mbox{and}\;}

\def\cd{d_{\text{cob}}}

\DeclareMathOperator{\dalt}{dalt}
\DeclareMathOperator{\alt}{alt}

\title[Khovanov width and dealternation]{Khovanov width and dealternation number of positive braid links}
\author{S.~Baader, P.~Feller, L.~Lewark, R.~Zentner}
\date{}
\thanks{The second and third author are grateful for support by the Swiss National Science Foundation and the Max Planck Institute for Mathematics. The fourth author is grateful for support by the SFB
`Higher Invariants' at the University of Regensburg, funded by the Deutsche Forschungsgemeinschaft (DFG)}
\begin{document}
\begin{abstract}
We give asymptotically sharp upper bounds for the Khovanov width and the dealternation number of positive braid links, in terms of their %
crossing number. The same braid-theoretic technique, combined with Ozsv\'ath, Stipsicz, and Szab\'o's Upsilon invariant, allows us to determine the exact cobordism distance between %
torus knots with braid index two and six.
\end{abstract}

\maketitle
\section{Introduction}

Every link diagram with $n$ crossings can be turned into one of the two alternating diagrams with the same underlying projection by changing at most $n/2$ crossings. Therefore the ratio between the dealternation number $\dalt(L)$ -- the smallest number of crossing changes needed to turn some diagram of $L$ into an alternating {diagram} -- and the %
crossing number $c(L)$ of a link $L$ is at most one-half.
We show that this ratio is bounded away from one-half for positive braid links with fixed braid index. The latter condition is necessary; we will exhibit a family of positive braid links with increasing braid index whose ratio $\dalt/c$ converges to one-half.

\begin{theorem} \label{main1}
Let $L$ be a link of %
braid index $n$ that can be represented as the closure of a positive braid on $n$ strands. Then
$$\frac{\dalt(L)}{c(L)} \leq \frac{1}{2}-\frac{1}{2(n^2-n+1)}.$$
\end{theorem}

The following result shows the asymptotic optimality of this ratio. Incidentally, it also settles the question about the largest possible ratio between the Khovanov width $w_{Kh}(L)$ of a link $L$ and its crossing number $c(L)$.

\begin{proposition} \label{propLn}
The family of links $L_n$ defined as the closures of the braids $\beta_n=(\sigma_1 \ldots \sigma_{n-1} \sigma_{n-1} \ldots \sigma_1)^{n-1}$ on $n$ strands satisfies
$$\lim_{n \to \infty} \frac{\dalt(L_n)}{c(L_n)}=\lim_{n \to \infty} \frac{w_{Kh}(L_n)}{c(L_n)}=\frac{1}{2}.$$
\end{proposition}

As discussed above, the ratio $\dalt(L)/c(L)$ cannot exceed one-half. Similarly, the ratio $w_{Kh}(L)/c(L)$ has no accumulation point above one-half, since the Khovanov width is bounded from above by the dealternation number (see~\cite[Theorem~8]{CK}):
\[w_{Kh}(L) \leq \dalt(L)+2.\]

At present, the question about the largest ratio $\dalt/c$ for positive braid links with fixed braid index $n$ remains open. %
However, the answer is known to be $\frac{1}{4}$ for $n=3$ by Abe and Kishimoto's work on $\dalt$ of 3--stranded braids~\cite{AK}; we determine the answer for $n=4$.

\begin{proposition} \label{prop4braids}
Let $L$ be a link of %
braid index $4$ that can be represented as the closure of a positive braid on 4 strands. Then
\[\frac{\dalt(L)}{c(L)} \leq \frac{1}{3}.\]
Moreover, the family of links defined as the closures of the $4$--braids\linebreak $(\sigma_1 \sigma_2 \sigma_3 \sigma_3 \sigma_2 \sigma_1)^n$ %
attains this bound in the limit $n \to \infty$.
\end{proposition}

Computations suggest that the ratio $w_{Kh}(L)/c(L)$ is far less than one-half for torus links $L=T(p,q)$. In fact, we expect their asymptotic ratio to be
\[\lim_{n \to \infty} \frac{w_{Kh}(T(n,n))}{c(T(n,n))}=\frac{1}{4}.\]
This would follow from the sharpness of Sto\v{s}i\'{c}'s inequality for the Khovanov width (\cite[Corollary 5]{St}; see \cref{stosic} below).
The following result provides evidence towards this; it shows that Sto\v{s}i\'{c}'s inequality is asymptotically sharp for %
torus links with braid index $6$. %

\begin{proposition} \label{prop6torus}
For all integers $n\geq3$ and $k \geq 1$:\\[-2ex]
\begin{enumerate}
\item[\emph{(i)}]   $\displaystyle \dalt(T(6,2n)) \leq 2n+2,$\\[-1ex]
\item[\emph{(ii)}]  $\displaystyle \dalt(T(6,2n+1)) \leq 2n+2,$\\[-1ex]
\item[\emph{(iii)}] $\displaystyle 6k \leq w_{Kh}(T(6,6k))-2\leq \dalt(T(6,6k)),$\\[-1ex]
\item[\emph{(iv)}]  $\displaystyle 6k-1\pm1 \leq \dalt(T(6,6k\pm1)),$\\[-1ex]
\item[\emph{(v)}]   $\displaystyle \lim_{k \to \infty} \frac{\dalt(T(6,6k))}{c(T(6,6k))}=\lim_{n \to \infty} \frac{w_{Kh}(T(6,6k))}{c(T(6,6k))}=\frac{1}{5}.$
\end{enumerate}
\end{proposition}

The proof of the upper bounds in \cref{prop6torus} consists in finding braid representatives of torus links with the smallest possible number of generators $\sigma_i$ with even index, i.e.~$\sigma_2$ or $\sigma_4$.
This technique for obtaining upper bounds has another interesting application to the smooth cobordism distance $\cd(K,L)$ of pairs of knots $K,L$, defined as the minimal genus among all smooth cobordisms in $S^3 \times [0,1]$ connecting $K \times \{0\}$ and $L \times \{1\}$. We write $\upsilon(K)$ for $\Upsilon_K(1)$, where $\Upsilon_K$ denotes the Upsilon invariant of a knot $K$, defined by Ozsv\'ath, Stipsicz and Szab\'o in~\cite{OSS_2014}, and we denote by $\tau(K)$ the tau invariant of a knot $K$ defined by Ozsv\'ath and Szab\'o in~\cite{OzsvathSzabo_03_KFHandthefourballgenus}.

\begin{theorem} \label{main3}
For torus knots $K$ and $L$ of braid index $2$ and $6$, respectively, we have
\[\cd(K,L)=\max\left\{|\upsilon(L)-\upsilon(K)|,|\tau(L)-\tau(K)|\right\}.\]
\end{theorem}
An explicit formula for $\cd(K,L)$
is provided after the proof of \cref{main3}; see~\eqref{eq:cobdistexplicit}.
All the statements concerning general positive braids and 4--braids are proved in the next section; the results about torus links are proved in \cref{sec:6strands}.
\cref{sec:altvsdalt} contains an analogue of \cref{prop6torus} for %
torus links with braid index $4$,
and compares the dealternation number with the alternation number.

\section{Twist regions and Khovanov width of positive braids}\label{sec:twistandwidth}

The proofs of \cref{main1} and \cref{propLn} involve an estimation of the crossing number and the dealternation number of positive braid links. The former task is easy, thanks to a result of Bennequin: if a link $L$ is represented by a positive braid whose number of strands coincides with the %
braid index of~$L$, then that braid realises the %
crossing number $c(L)$. Indeed, the canonical Seifert surface associated with the closure of a positive braid has minimal genus (see~\cite{B}); a diagram with fewer crossings and at least as many Seifert circles would result in a Seifert surface of smaller genus, a contradiction. Here we recall that the number of Seifert circles is not smaller than the %
braid index of a link (see~\cite{Y}). For the second task, we need an upper bound for the dealternation number in terms of the number of twist regions $t$ of a positive braid representing a link $L$. A twist region of an $n$--braid is a maximal subword of the form $\sigma_i^k$, for some generator $\sigma_i$ in the braid group on $n$ strands. %
The following inequality was proved by Abe and Kishimoto (\cite[Lemma 2.2]{AK}; the generalisation from $3$--braids to $n$--braids is straightforward, see \cref{fig:twist}):
\[\dalt(L) \leq \frac{t}{2}.\]
\begin{figure}[t]
\centering
\includegraphics{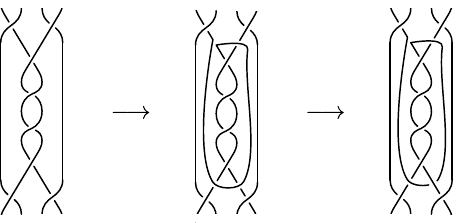}
\caption{How to alternate around one twist region with one crossing change.}
\label{fig:twist}
\end{figure}

\begin{proof}[Proof of \cref{main1}]
Let $\beta$ %
be a positive $n$--braid whose closure is a link $L$ of %
braid index $n$. We write $\beta$ as a product of positive braids $\beta_1 \ldots \beta_k \alpha$, where all $\beta_i$ have $\frac{1}{2}n(n-1)+1$ crossings, and $\alpha$ has strictly less crossings (the case $k=0$, i.e. $\beta=\alpha$, is also allowed). The condition on the number of crossings guarantees that every braid $\beta_i$ has two strands that cross at least two times. Consider an innermost bigon formed by two such strands. Then all other strands intersecting that bigon pass over it from the bottom left to the top right, or pass under it from the bottom right to the top left (see \cref{fig:digon}).

\begin{figure}[t]
\begin{minipage}[b]{.4\textwidth}
\centering
\begin{asy}
drawbraid("bacbacb");
\end{asy}
\captionof{figure}{A bigon}
\label{fig:digon}
\end{minipage}%
\begin{minipage}[b]{.6\textwidth}
\centering
\begin{asy}
pen[] strandPens = {scale(2.5) * braidpen, braidpen, braidpen, braidpen, braidpen};
drawbraid("abdacba", strandPens);
\end{asy}
\quad\raisebox{1.45cm}{$\quad\longrightarrow\quad$}
\begin{asy}
pen[] strandPens = {scale(2.5) * braidpen, braidpen, braidpen, braidpen, braidpen};
drawbraid("bdcbabc", strandPens);
\end{asy}
\captionof{figure}{A reducible braid}
\label{fig:markov}
\end{minipage}
\end{figure}

\begin{figure}[t]
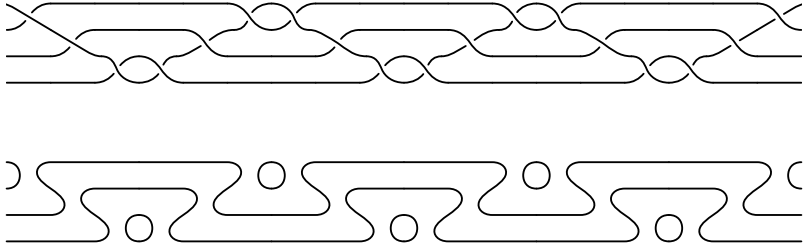

\begin{asy}
drawbraid("abccba abccba abccba");
shipout(Seascape);
\end{asy}
\vspace{6ex}
\begin{asy}
void stueck(transform t) {
mydraw(t * ((0,0){up}..(0.5,0.3)..{down}(1,0)));
mydraw(t * ((2,0)--(2,1){up}..(1.5,1.3)..(1,1)..(0.5,0.7)..{up}(0,1)--(0,3)));
mydraw(t * ((3,0)--(3,2){up}..(2.5,2.3)..(2,2)..(1.5,1.7)..{up}(1,2)--(1,3)));
mydraw(t * ((2,3){down}..(2.5,2.7)..{up}(3,3)));
}
for (int i = 0; i < 3; ++i) {
stueck(shift(0, 6 * i));
stueck(shift(0, 3 + 6 * i) * reflect((1.5,0),(1.5,1)));
}
shipout(Seascape);
\end{asy}
\caption{The braid $(\sigma_1\sigma_2\sigma_3\sigma_3\sigma_2\sigma_1)^3$ and its all-B smoothing.}
\label{fig:Bsmoothing}
\end{figure}
These strands can be moved away by an isotopy, giving rise to a positive braid containing a square of a generator. Altogether, we obtain a positive braid equivalent to $\beta$ with at least $k$ squares of generators. By Abe and Kishimoto's result, the dealternation number of $L$ is at most one-half times the number of twist regions of that braid:
$$\dalt(L) \leq \frac{1}{2}(c(L)-k).$$
If $k \geq 1$, the highest possible ratio $\dalt(L)/c(L)$ comes from the case $k=1$, $c(L)=n(n-1)+1$, $\dalt(L) \leq \frac{1}{2}n(n-1)$; it is
$$\frac{\dalt(L)}{c(L)} \leq \frac{1}{2}-\frac{1}{2(n^2-n+1)},$$
as desired. If $k=0$, i.e. $\beta=\alpha$, then either $\alpha$ contains a bigon, leading to a lower ratio $\dalt(L)/c(L)$, or $\alpha$ can be reduced by a Markov move. Indeed, in the latter case, the strand starting at the bottom left of $\alpha$ crosses some number of strands before reaching the top. It can therefore be moved to the top of $\alpha$ and then reduced by a Markov move (see \cref{fig:markov}), contradicting the assumption on the minimality of the braid index of $\beta$.
\end{proof}

\begin{proof}[Proof of \cref{propLn}]
The links $L_n$ represented by the family of braids $\beta_n$ %
have $n$ components. Therefore, their %
braid index is $n$. By the above remark, their %
crossing number is realised by the braids $\beta_n$: $c(L_n)=2(n-1)^2$.
The key observation needed to compute the Khovanov width $w_{Kh}(L_n)$ is the adequacy of the diagrams obtained by closing the braids $\beta_n$. This means by definition that the all-A smoothing and the all-B smoothing of crossings results in a union of circles that have no points of self-contact. In the case of our braids~$\beta_n$, this is easy to check, since the all-A and all-B smoothings of positive braid diagrams correspond to all vertical and all horizontal smoothing, respectively. The Khovanov width of a link $L$ with an adequate diagram $D$ can then be determined by another result of Abe (\cite[Theorem 3.2]{A}; the generalisation from knots to multi-component links is straightforward):
$$w_{Kh}(L)=\frac{1}{2}(c(D)-s_A(D)-s_B(D))+3.$$
Here $s_A(D)$ and $s_B(D)$ denote the number of circles resulting from the all-A and the all-B smoothings of $D$, respectively. We compute $s_A=n$, $s_B=3n-4$ for the closures of the braids $\beta_n$ (see \cref{fig:Bsmoothing}) and deduce
$$\lim_{n \to \infty} \frac{w_{Kh}(L_n)}{c(L_n)}=\lim_{n \to \infty} \frac{\dalt(L_n)}{c(L_n)}=\frac{1}{2}.$$
For the latter, we recall $w_{Kh}(L) \leq \dalt(L)+2$ and $\dalt(L) \leq c(L)/2$.
\end{proof}
\begin{proof}[Proof of \cref{prop4braids}]
Choose a positive $4$--braid $\beta$ %
representing $L$ with the minimal number of generators of type $\sigma_2$, and conjugate it so that it does not start with a generator $\sigma_2$. Then two consecutive twist regions of the form $\sigma_2^k$ are separated by at least two crossings of type $\sigma_1$ or $\sigma_3$. This is also true for the last and first twist region, when viewing these as consecutive along the closed braid. Therefore, the number of twist regions of the form $\sigma_2^k$ is at most a third of the number of crossings of $\beta$. We conclude as in the proof of \cref{main1}. For the second statement, we observe that the links $L_n$ defined as the closures of the $4$--braids $(\sigma_1 \sigma_2 \sigma_3 \sigma_3 \sigma_2 \sigma_1)^n$ %
are again adequate, which allows for a simple computation of their Khovanov width. The resulting limits are
\[
\pushQED{\qed}
\lim_{n \to \infty} \frac{\dalt(L_n)}{c(L_n)}=\lim_{n \to \infty} \frac{w_{Kh}(L_n)}{c(L_n)}=\frac{1}{3}.
\qedhere
\popQED
\]
\renewcommand{\qedsymbol}{}
\end{proof}
\section{Dealternation number and cobordism distance for torus links with braid index 6}\label{sec:6strands}

The following braid-theoretic observation is the main geometric input for \cref{main3} and \cref{prop6torus}.
\begin{lemma}\label{lemma:manyodd}
For all integers $n\geq 0$, there exists a positive $6$--braid word $\beta_n$ with $8n+3$ odd generators (i.e.~$\sigma_1$, $\sigma_3$, and $\sigma_5$) and $2n+2$ even generators (i.e.~$\sigma_2$ and $\sigma_4$) such that $\beta_n$ represents the standard torus link $6$--braid $(\sigma_1\sigma_2\sigma_3\sigma_4\sigma_5)^{2n+1}$.
\end{lemma}
\begin{proof}
The case $n=0$ is trivial since $(\sigma_1\sigma_2\sigma_3\sigma_4\sigma_5)$ contains 3 odd and 2 even generators.
For the case $n=1$,
observe that the positive braid given by the $6$--braid word $(\sigma_1\sigma_2\sigma_3\sigma_4\sigma_5)^{3}$ is isotopic to
\begin{multline}\label{eq:iso1}
\sigma_1\sigma_3(\sigma_1\sigma_2\sigma_3\sigma_4\sigma_5)(\sigma_1\sigma_2\sigma_3\sigma_4)(\sigma_2\sigma_3\sigma_4\sigma_5)
=\\
\sigma_1\sigma_3(\sigma_1\sigma_2\sigma_3\sigma_4\sigma_5)(\sigma_1\sigma_2\sigma_3\sigma_2\sigma_4\sigma_3\sigma_4\sigma_5);
\end{multline}
compare \cref{fig:torus}.
\begin{figure}[t]
\begin{minipage}[b]{.45\textwidth}
\centering
\begin{asy}
drawbraid("abcde abcde abcde");
\end{asy}
\quad\raisebox{2.8cm}{$\ \longrightarrow\,$}
\begin{asy}
pen r = gray(0.7);
filldraw(myt * shift(2.6, 7.5) * inverse(myt) * scale(1.2,2.2) * unitcircle, r, white);
filldraw(myt * shift(4.4, 9.5) * inverse(myt) * scale(1.2,2.2) * unitcircle, r, white);
drawbraid("ac abcde abcbdcde");
\end{asy}
\captionsetup{width=.8\textwidth}
\captionof{figure}{The isotopy from $(\sigma_1\sigma_2\sigma_3\sigma_4\sigma_5)^3$ to \cref{eq:iso1}.}
\label{fig:torus}
\end{minipage}%
\hfill
\begin{minipage}[b]{.45\textwidth}
\centering
\begin{asy}
pen r = gray(0.7);
filldraw(myt * shift(3.5, 7 + 1.5) * inverse(myt) * scale(0.8) * unitcircle, r, white);
filldraw(myt * shift(5.5, 8 + 1.5) * inverse(myt) * scale(0.85) * (((N..(0.4*(N+E))..E) & (E..(0.4*(E+S))..S) & (S..(0.4*(S+W))..W) & (W..(0.4*(W+N))..N))--cycle), r, white);
drawbraid("abcde abcde ec");
\end{asy}
\quad\raisebox{2.1cm}{$\ \longrightarrow\,$}
\begin{asy}
pen r = gray(0.7);
filldraw(myt * shift(5.5, 2.5) * inverse(myt) * scale(0.8) * unitcircle, r, white);
filldraw(myt * shift(1.5, 2.5) * inverse(myt) * scale(0.85) * (((N..(0.4*(N+E))..E) & (E..(0.4*(E+S))..S) & (S..(0.4*(S+W))..W) & (W..(0.4*(W+N))..N))--cycle), r, white);
drawbraid("ae abcde abcde");
\end{asy}
\captionsetup{width=.8\textwidth}
\captionof{figure}{The braid isotopy \cref{eq:T62commuteswithodd} for $i = 3,5$.}
\label{fig:commute}
\end{minipage}
\end{figure}
By applying $\sigma_2\sigma_3\sigma_2\sigma_4\sigma_3\sigma_4=\sigma_3\sigma_2\sigma_3\sigma_3\sigma_4\sigma_3$ (indicated in grey in \cref{fig:torus}), we find
\begin{equation}\label{eq:T63}(\sigma_1\sigma_2\sigma_3\sigma_4\sigma_5)^3=\sigma_1\sigma_3(\sigma_1\sigma_2\sigma_3\sigma_4\sigma_5)(\sigma_1\sigma_3\sigma_2\sigma_3\sigma_3\sigma_4\sigma_3\sigma_5).\end{equation}
The right-hand side of~\eqref{eq:T63} can be taken to be $\beta_1$, since it has $11$ odd generators and $4$ even generators.

Next we consider the case $n\geq 2$. We reduce this to the case $n=1$ by using that odd generators `commute' with $(\sigma_1\sigma_2\sigma_3\sigma_4\sigma_5)^2$ as follows:
\begin{equation}\label{eq:T62commuteswithodd}
(\sigma_1\sigma_2\sigma_3\sigma_4\sigma_5)^2\sigma_i=\sigma_{i+2}(\sigma_1\sigma_2\sigma_3\sigma_4\sigma_5)^2\end{equation}
for all $i$ in $\{1,3,5\}$, where $i+2$ is read modulo $6$ (compare \cref{fig:commute}). %
Using Equations~\cref{eq:T63} and \cref{eq:T62commuteswithodd}, we rewrite $(\sigma_1\sigma_2\sigma_3\sigma_4\sigma_5)^{2n+1}$ as
\begin{align*}
& (\sigma_1\sigma_2\sigma_3\sigma_4\sigma_5)^{2n+1} \\
\overset{\text{\eqref{eq:T63}}}{=}\ &
(\sigma_1\sigma_2\sigma_3\sigma_4\sigma_5)^{2n-2}\sigma_1\sigma_3(\sigma_1\sigma_2\sigma_3\sigma_4\sigma_5)(\sigma_1\sigma_3\sigma_2\sigma_3\sigma_3\sigma_4\sigma_3\sigma_5)\\
\overset{\text{\eqref{eq:T62commuteswithodd}}}{=}\ &
\sigma_i\sigma_{i+2}(\sigma_1\sigma_2\sigma_3\sigma_4\sigma_5)^{2n-2}(\sigma_1\sigma_2\sigma_3\sigma_4\sigma_5)(\sigma_1\sigma_3\sigma_2\sigma_3\sigma_3\sigma_4\sigma_3\sigma_5)\\
=\hspace{1pt}\ &\sigma_i\sigma_{i+2}(\sigma_1\sigma_2\sigma_3\sigma_4\sigma_5)^{2n-1}(\sigma_1\sigma_3\sigma_2\sigma_3\sigma_3\sigma_4\sigma_3\sigma_5),\end{align*}
where $i=1$, $i=3$, or $i=5$; depending on whether $n$ is $0$, $1$, or $2$ modulo 3. Again $i+2$ is read modulo $6$.
Applying the above inductively to $(\sigma_1\sigma_2\sigma_3\sigma_4\sigma_5)^{2l+1}$ for $l\leq n$, we find
\begin{align*}
(\sigma_1\sigma_2\sigma_3\sigma_4\sigma_5)^{2n+1}
&=\sigma_i\sigma_{i+2}(\sigma_1\sigma_2\sigma_3\sigma_4\sigma_5)^{2n-1}(\sigma_1\sigma_3\sigma_2\sigma_3\sigma_3\sigma_4\sigma_3\sigma_5)\\
&=\sigma_i\sigma_{i+2}\sigma_{i-2}\sigma_{i}(\sigma_1\sigma_2\sigma_3\sigma_4\sigma_5)^{2n-3}(\sigma_1\sigma_3\sigma_2\sigma_3\sigma_3\sigma_4\sigma_3\sigma_5)^2\\
&=\cdots\\
&=\sigma_1^{k_1}\sigma_3^{k_3}\sigma_5^{k_5}(\sigma_1\sigma_2\sigma_3\sigma_4\sigma_5)(\sigma_1\sigma_3\sigma_2\sigma_3\sigma_3\sigma_4\sigma_3\sigma_5)^{n},
\end{align*}
where $k_1+k_2+k_3=2n$.
\end{proof}

\cref{lemma:manyodd} has an interesting application concerning fibre surfaces of braid index $6$ %
torus knots, which we will use in the proof of Theorem 2. Let $F(p,q)$ denote the unique fibre surface of the torus link $T(p,q)$.

\begin{proposition}\label{prop:subsurfaces}
For all integers $n\geq 2$, the fibre surface $F(6,2n+1)$ contains $F(2,8n+1)$ as an incompressible subsurface.
In particular, \begin{align*}
\cd(T(6,6k+1),T(2,24k+1))&=g(T(6,6k+ 1))-g(T(2,24k+1))\et\\
\cd(T(6,6k-1),T(2,24k-7))&=g(T(6,6k- 1))-g(T(2,24k-7)).
                             \end{align*}
for all positive integers $k$, where $g(T(p,q))=\frac{1}{2}(p-1)(q-1)$ denotes the Seifert genus of $T(p,q)$ for positive coprime integers $p, q$.
\end{proposition}
\begin{proof}[Proof of \cref{prop:subsurfaces}]
To the closure of a positive braid word $\beta$, we associate its canonical Seifert surface given by vertical disks for every strand and half twisted bands connecting them for every generator in $\beta$. As remarked in \cref{sec:twistandwidth}, this is a minimal genus Seifert surface. In particular, the $6$--strand positive braid word $(\sigma_1\sigma_2\sigma_3\sigma_4\sigma_5)^{2n+1}$ yields the fibre surface $F(6,2n+1)$.
We rewrite $(\sigma_1\sigma_2\sigma_3\sigma_4\sigma_5)^{2n+1}$ as
$(\sigma_1\sigma_2\sigma_3\sigma_4\sigma_5)^2(\sigma_1\sigma_2\sigma_3\sigma_4\sigma_5)^{2(n-1)+1}$ and then apply \cref{lemma:manyodd} to find a braid word $\beta_{n-1}$ with $2n$ even generators and $8n-5$ odd generators such that
\[(\sigma_1\sigma_2\sigma_3\sigma_4\sigma_5)^{2n+1}=(\sigma_1\sigma_2\sigma_3\sigma_4\sigma_5)^2\beta_{n-1}.\]
By deleting the $2n$ even generators in $\beta_{n-1}$, we find a positive braid word
\[\alpha_n=(\sigma_1\sigma_2\sigma_3\sigma_4\sigma_5)^2\sigma_1^{k_1}\sigma_3^{k_3}\sigma_5^{k_5},\] where
$k_1$, $k_3$ and $k_5$ are positive integers such that $k_1+k_3+k_5=8n-5$. The closure of $\alpha_n$ is the torus knot $T(2,8n+1)$. Since deleting a generator in a positive braid word corresponds to deleting a band in the associated Seifert surface, we have that $F(6,2n+1)$ may be turned into $F(2,8n+1)$ by removing $2n$ bands. Consequently, $F(2,8n+1)$ is an incompressible subsurface of $F(6,2n+1)$.

For the second statement of the Proposition, we recall that, if a knot $K$ is the boundary of a genus $g_K$ incompressible subsurface of a genus $g_L$ Seifert surface with boundary the knot $L$, then there exists a cobordism of genus $g_L-g_K$ between $K$ and $L$. Applying this to $T(2,8n+1)$ and $T(6,2n+1)$ for $n=3k\pm1$ yields a cobordism of genus $n$. More explicitly, such a cobordism is e.g.~given by $2n$ saddles guided by the $2n$ bands corresponding to the deleted generators described in the previous paragraph. Finally, $n$ realises the cobordism distance since %
by the triangle inequality the cobordism distance
$\cd(T(6,2n+1),T(2,8n+1))$ is greater than or equal to
\[\cd(T(6,2n+1),\mathrm{unknot})-\cd(T(2,8n+1,\mathrm{unknot})=5n-4n,\]%
where the equality is given by the local Thom conjecture proven by Kronheimer and Mrowka~\cite[Corollary~1.3]{KronheimerMrowka_Gaugetheoryforemb}:
\[\cd(T(p,q),\mathrm{unknot})=%
\frac{(p-1)(q-1)}{2}\text{ for all coprime } p,q\geq 1.\qedhere\]
\end{proof}
For the proofs of \cref{prop6torus} and \cref{main3}, we use
\cref{lemma:manyodd} and
\cref{prop:subsurfaces} as geometric inputs, respectively. As an obstruction to cobordisms and the dealternation number we use the Upsilon invariant, which we recall next, before applying it in the proofs of \cref{prop6torus} and \cref{main3}.

In~\cite{OSS_2014}, Ozsv\'ath, Stipsicz, and Szab\'o introduced an infinite family of concordance invariants $\Upsilon(t)$, parametrised by the interval $[0,2]$. We use $\upsilon$ -- the invariant corresponding to $t=1$ -- and the $\tau$--invariant as introduced by Ozsv\'ath and Szab\'o in~\cite{OzsvathSzabo_03_KFHandthefourballgenus}. The latter can be recovered as $\lim_{t\to0}\frac{-\Upsilon(t)}{t}$.

Both $\tau$ and $\upsilon$ are integer-valued concordance invariants. In fact, they both bound the smooth slice genus and, thus, the cobordism distance of knots~\cite[Corollary~1.3]{OzsvathSzabo_03_KFHandthefourballgenus}\cite[Theorem~1.11]{OSS_2014}.
Thus, for all knots $K$ and $L$ we have
\begin{equation}\label{eq:lowerboundcob}
|\upsilon(L)-\upsilon(K)|,|\tau(L)-\tau(K)|\leq \cd(K,L).
\end{equation}

As a consequence of the fact that $\upsilon$ equals $-\tau$ on alternating knots and their similar behaviour under crossing changes, one has for all knots $K$ (compare~\cite[Corollary~3]{FellerPohlmannZentner_15}):
\begin{equation}\label{eq:lowerbounddalt}
|\tau(K)+\upsilon(K)|\leq \dalt(K).
\end{equation}

The $\tau$--invariant equals the genus of positive torus knots~\cite[Corollary~1.7]{OzsvathSzabo_03_KFHandthefourballgenus}.
We recall the value of $\upsilon$ on positive torus knots of braid index $2$ and $6$.
\begin{lemma}\label{lemma:upsilonfortorusknots}
For all positive integers $k$,
\begin{align*}
\upsilon(T(2,2k+1)) &=-k,\\
\upsilon(T(6,6k+1)) &=-9k,\\
\upsilon(T(6,6k+5)) &=-9k-6.
\end{align*}
\end{lemma}
\begin{proof}The values of $\upsilon$ for torus knots with braid index $2$ (or more generally thin knots) are provided in~\cite[Theorem~1.14]{OSS_2014}. For torus knots of braid index $6$, the inductive formula from~\cite[Proposition~2.2]{FellerKrcatovich_16_OnCobBraidIndexAndUpsilon} yields
\begin{align*}
\upsilon(T(6,6k+1))&=k\upsilon(T(6,7))=-9k& \et\\
\upsilon(T(6,6k+5))&=k\upsilon(T(6,7))+\upsilon(T(6,5))=-9k-6.&&\qedhere
\end{align*}
\end{proof}

\begin{proof}[Proof of \cref{prop6torus}]
Items $(i)$ and $(ii)$ follow from \cref{lemma:manyodd}. Indeed, by \cref{lemma:manyodd}, there exists a positive braid word $\beta_n$ with closure $T(6,2n+1)$ that has $2n+2$ even generators. Changing the corresponding $2n+2$ positive crossing to negative crossings in the associated diagram for $T(6,2n+1)$ yields an alternating diagram. Thus, we have $\dalt(T(6,2n+1))\leq 2n+2$.
Similarly, by \cref{lemma:manyodd}, the torus link $T(6,2n)$ is the closure of a positive braid word $\beta_{n-1}(\sigma_1\sigma_2\sigma_3\sigma_4\sigma_5)$, which has $2n+2$ even generators. By the same reasoning as above this yields $\dalt(T(6,2n))\leq 2n+2$.

The lower bound for the Khovanov width claimed in $(iii)$ is given by Sto\v{s}i\'{c}'s inequality (\cite[Corollary 5]{St}),
\begin{equation}\label{stosic}
	w_{Kh}(T(2n,2kn)) \geq n(n-1)k + 2 \, .
\end{equation}

The lower bound claimed in $(iv)$ follows from~\eqref{eq:lowerbounddalt}. Indeed, by \cref{lemma:upsilonfortorusknots}, we have $\upsilon(T(6,6k+1))=-9k$ and, therefore,
\[6k=15k-9k=|\tau(T(6,6k+1))+\upsilon(T(6,6k+1))|\leq \dalt(T(6,6k+1)).\]
Similarly, we have
\begin{equation*}
\begin{split}
6k-2=15k-5-9k+3 & =|\tau(T(6,6k-1))+\upsilon(T(6,6k-1))| \\
& \leq \dalt(T(6,6k-1)).
\end{split}
\end{equation*}

Finally, $(v)$ follows from $(i)$ and $(iii)$ since $c(T(6,6k))=30k$ (compare with the beginning of \cref{sec:twistandwidth}) and $w_{Kh}\leq \dalt+2$.
\end{proof}

Next we turn to the cobordism distance between torus knots of braid index 2 and torus knots of braid index 6.
In fact, it will be clear from the proof below that $\cd(K,L)=\cd(K,J)+\cd(J,L)$, where $J$ is the (unique) braid index 2 torus knot of maximal genus such that $\cd(J,L)=g(L)-g(J)$. See~\eqref{eq:cobdistexplicit} below for an explicit formula for $\cd(K,L)$. %
\begin{proof}[Proof of \cref{main3}]
For the entire proof, we write $L=T(6,m)$ and $K=T(2,n)$, where $n$ is an odd integer and $m$ is an integer coprime to 6. Also, by taking mirror images, we may (and do) assume that $m$ is positive. Furthermore, we take $J$ to be the positive torus knot $T(2,4(m-1)+1)$.
Note that, by \cref{prop:subsurfaces}, there exists a cobordism of genus
\begin{equation}\label{eq:optcob}g(L)-g(J)=\tau(L)-\tau(J)=-\upsilon(J)+\upsilon(L),\end{equation}
where the last equality follows immediately from \cref{lemma:upsilonfortorusknots}.

Let us first consider the case $n\leq 4(m-1)+1$. Then
\begin{equation}\label{eq:cobdistKJ}\cd(K,J)=\tau(J)-\tau(K)=\left\{\begin{array}{cc}g(J)-g(K)&\text{if }n>0\\g(J)+g(K)&\text{if }n<0\end{array}\right..\end{equation}
Therefore,
\begin{align*}
\cd(K,L) & =\cd(K,J)+\cd(J,L) \\
     & \overset{\makebox[0pt][c]{\scriptsize\eqref{eq:cobdistKJ}\eqref{eq:optcob}}}{=}\ \tau(L)-\tau(K)=\left\{\begin{array}{cc}g(L)-g(K)&\text{if }n>0\\g(L)+g(K)&\text{if }n<0\end{array}\right..
\end{align*}
Indeed, we have $\cd(K,L)\leq\cd(K,J)+\cd(J,L)$ by composition of cobordisms and $\cd(K,L)\geq\cd(K,J)+\cd(J,L)$ follows from the fact that $\tau(L)-\tau(K)$ is a lower bound for $\cd(K,L)$.

This leaves us with the case $n>4(m-1)+1$.
Similarly to~\eqref{eq:cobdistKJ} we have
\begin{equation}\label{eq:cobdistKJagain}\cd(K,J)=-\upsilon(K)+\upsilon(J)=g(K)-g(J).\end{equation}
Thus,
\begin{align*}
\cd(K,L) & =\cd(K,J)+\cd(J,L)\\
& \overset{\makebox[0pt][c]{\scriptsize\eqref{eq:cobdistKJagain}\eqref{eq:optcob}}}{=}\
-\upsilon(K)+\upsilon(J)-\upsilon(J)+\upsilon(L)\\
& =-\upsilon(K)+\upsilon(L).
\end{align*}
Indeed, we have $\cd(K,L)\leq\cd(K,J)+\cd(J,L)$ by composition of cobordisms and $\cd(K,L)\geq\cd(K,J)+\cd(J,L)$ follows from the fact that $-\upsilon(K)+\upsilon(L)$ is a lower bound for $\cd(K,L)$.
\end{proof}

If we choose $K=T(2,n)$ and $L=T(6,m)$ as in the above proof, where $n$ is an odd integer and $m$ is an integer coprime to 6 and $m\geq 7$, then the distance from \cref{main3} can be explicitly given by
\begin{equation}\label{eq:cobdistexplicit}\cd(K,L)=\frac{|4(m-1)+1-n|}{2}+\frac{m-1}{2}.\end{equation}

\section{Alternation number and torus links of braid index 4}\label{sec:altvsdalt}
We briefly comment on the \emph{alternation number} $\alt(L)$ of a link $L$ -- the smallest number of crossing changes needed to make $L$ alternating. We note that $\alt(L)$ is different from $\dalt(L)$ -- the smallest number of crossing changes needed to turn some diagram of $L$ into an alternating \emph{diagram}. Clearly, $\alt(L)\leq\dalt(L)$ for all links. This inequality can be strict. The latter follows for example from the fact that all Whitehead doubles $W(K)$ of a knot $K$ satisfy $\alt(W(K))\leq 1$, while $w_{Kh}(W(K))\leq\dalt(W(K))+2$ can be arbitrarily large. While the lower bound given by $w_{Kh}$ no longer holds for $\alt$, the lower bound given by $|\tau+\upsilon|$ still holds; compare~\cite[Corollary~3]{FellerPohlmannZentner_15}.

Consequently all upper bounds for $\dalt$ provided in this paper also hold for $\alt$ and, for torus knots of braid index $6$, the alternation number, like the dealternation number, is determined up to an ambiguity of $2$. Indeed,
\[6k-1\pm1 \leq \alt(T(6,6k\pm1))\leq 6k+1\pm 1,\]
by \cref{prop6torus} (and its proof).

Let us conclude by discussing the case of %
braid index $4$ torus links. While an analogue of \cref{lemma:manyodd} holds, the consequences for $\dalt$ are less interesting since $\alt$ was previously determined for $T(4,2n+1)$~\cite[Theorem~1]{FellerPohlmannZentner_15}. Also, the analogue of \cref{main3} (and \cref{prop:subsurfaces}) was previously established; compare~\cite[Corollary~3 (and Theorem~2)]{Feller_15_MinCobBetweenTorusknots}. We briefly summarise the results that can be obtained by the same techniques we used in \cref{sec:6strands}.

\begin{lemma}
For all integers $n\geq0$, there exists a positive $4$--braid word $\beta_n$ with $5n+2$ odd generators and $n+1$ even generators such that $\beta_n$ represents the $4$--braid $(\sigma_1\sigma_2\sigma_3)^{2n+1}$.\qed
\end{lemma}
As a consequence one finds
\begin{align}
n\leq\dalt(T(4,2n+1)) & \leq n+1\notag\qquad\et\\
\dalt(T(4,2n))& \leq n+1;\label{eq:dalttorus4strand}
\end{align}
in comparison, one has $\alt(T(4,2n+1))=n$~\cite[Theorem~1]{FellerPohlmannZentner_15}.
By Sto\v{s}i\'{c}'s inequality and $c(T(4,4k))=12k$, \eqref{eq:dalttorus4strand} yields
\begin{corollary}\label{dalt_braid-index_4}
For all integers $k\geq 1$,
\[2k\leq w_{Kh}(T(4,4k))-2\leq\dalt(T(4,4k))\leq 2k+1\]
and
\[
\lim_{n \to \infty} \frac{w_{Kh}(T(4,4n))}{c(T(4,4n))}= \frac{\dalt(T(4,4n))}{c(T(4,4n))} = \frac{1}{6}.\]
\end{corollary}

\bigskip
\newcommand{\myemail}[1]{\texttt{\href{mailto:#1}{#1}}}
\myemail{sebastian.baader@math.unibe.ch}

\myemail{peter.feller@math.ch}

\myemail{lukas@lewark.de}

\myemail{raphael.zentner@mathematik.uni-regensburg.de}


\begin{thebibliography}{[10]}
\normalsize
\bibitem{A}
     T.~Abe: \emph{The Turaev genus of an adequate knot}, Topology Appl.~\textbf{156} (2009), no.~17, 2704--2712.

\bibitem{AK}
     T.~Abe and K.~Kishimoto: \emph{The dealternation number and the alternation number of a closed 3--braid}, J. Knot Theory Ramifications~\textbf{19} (2010), no.~9, 1157--1181.

\bibitem{B}
     D.~Bennequin: \emph{Entrelacements et \'equations de Pfaff}, Ast\'erisque~\textbf{107}--\textbf{108} (1983), 87--161.

\bibitem{CK}
     A.~Champanerkar and I.~Kofman: \emph{Spanning trees and Khovanov homology}, Proc. Amer. Math. Soc.~\textbf{137} (2009), no.~6, 2157--2167.

\bibitem{Feller_15_MinCobBetweenTorusknots}
P.~Feller: \emph{Optimal cobordisms between torus knots}, Comm. Anal. Geom.~\textbf{24} (2016), no.~5, 993--1025.


\bibitem{FellerKrcatovich_16_OnCobBraidIndexAndUpsilon}
P.~Feller and D.~Krcatovich:
\emph{On cobordisms between knots, braid index, and the
  {U}psilon-invariant}, Math. Ann.~\textbf{369} (2017), no. 1-2, 301--329.

\bibitem{FellerPohlmannZentner_15}
P.~Feller, S.~Pohlmann, and R.~Zentner:
\emph{Alternating numbers of torus knots with small braid index}, Indiana Univ. Math.~J.~\textbf{67} (2018), no.~2, 645--655.

\bibitem{KronheimerMrowka_Gaugetheoryforemb}
P.~B.~Kronheimer and T.~S.~Mrowka:
\emph{Gauge theory for embedded surfaces~{I}},
Topology~\textbf{32} (1993), no.~4, 773--826.

\bibitem{OzsvathSzabo_03_KFHandthefourballgenus}
P.~Ozsv{\'a}th and Z.~Szab{\'o}:
\emph{ Knot {F}loer homology and the four-ball genus}, Geom. Topol.~\textbf{7} (2003), 615--639.

\bibitem{OSS_2014}
P.~Ozsv{\'a}th, A.~I.~Stipsicz, and Z.~Szab{\'o}: \emph{Concordance homomorphisms from knot {F}loer homology},
Adv. Math.~\textbf{315} (2017), 366--426.


\bibitem{St}
M.~Sto\v{s}i\'c: \emph{Khovanov homology of torus links}, Topology Appl.~\textbf{156} (2009), no.~3, 533--541.

\bibitem{Y}
S.~Yamada: \emph{The minimal number of Seifert circles equals the braid index of a link}, Invent. Math.~\textbf{89} (1987), no.~2, 347--356.


\end{thebibliography}
\end{document}